\documentclass[a4paper,11pt,reqno]{amsart}
\usepackage{graphicx, psfrag, amsmath, amscd, amssymb}
\newtheorem{lem}{Lemma}[section]
\newtheorem{thm}[lem]{Theorem}
\newtheorem{pro}[lem]{Proposition}
\newtheorem{cor}[lem]{Corollary}

\newtheorem{defi}[lem]{Definition}

\numberwithin{equation}{section}

\newcommand{\ZZ}{{\mathbb{Z}}}

\newcommand{\Des}{{\textsf{Des}}}
\newcommand{\des}{{\textsf{des}}}

\newcommand{\inv}{{\textsf{inv}}}

\newcommand{\maj}{{\textsf{maj}}}

\newcommand{\fmaj}{{\textsf{fmaj}}}

\newcommand{\Dmaj}{{\mathsf{Dmaj}}}
\newcommand{\sign}{{\textsf{sign}}}
\newcommand{\col}{{\mathsf{col}}}

\newcommand{\z}{{\mathbf{z}}}

\title[Signed Mahonian on Colored Permutation Groups]{Signed Mahonian on Parabolic Quotients of\\ Colored Permutation Groups}

\author[{S.-P. Eu}]{Sen-Peng Eu}
\address{Department of Mathematics, National Taiwan Normal University, Taipei 11677, and Chinese Air Force Academy, Kaohsiung 82047, Taiwan, ROC}
\email{speu@math.ntnu.edu.tw}

\author[T.-S. Fu]{Tung-Shan Fu}
\address{Department of Applied Mathematics, National Pingtung University, Pingtung 90003, Taiwan, ROC}
\email{tsfu@mail.nptu.edu.tw}

\author[Y.-H. Lo]{Yuan-Hsun Lo}
\address{Department of Applied Mathematics, National Pingtung University, Pingtung 90003, Taiwan, ROC}
\email{yhlo@mail.nptu.edu.tw}



\begin{document}

\begin{abstract} We study the generating polynomial of the flag major index with each one-dimensional character, called signed Mahonian polynomial, over the colored permutation group, the wreath product of a cyclic group with the symmetric group. Using the insertion lemma of Han and Haglund--Loehr--Remmel and a signed extension established by Eu \emph{et al.}, we derive the signed Mahonian polynomial over the quotients of parabolic subgroups of the colored permutation group, for a variety of systems of coset representatives in terms of subsequence restrictions. This generalizes the related work over parabolic quotients of the symmetric group due to Caselli as well as to Eu \emph{et al.} As a byproduct, we derive a product formula that generalizes Biagioli's result about the signed Mahonian on the even signed permutation groups.
\end{abstract}

\maketitle

\section{Introduction} The enumeration of the major index with sign over the symmetric group $S_n$ studied by Gessel and Simion (see \cite[Corollary 2]{Wachs}) yields a remarkable factorization formula in terms of $q$-factorials, called \emph{signed Mahonian polynomial},
\begin{equation} \label{eqn:Gessel-Simion}
\sum_{\sigma\in S_n} \sign(\sigma)q^{\maj(\sigma)}=[1]_{q}[2]_{-q}\cdots [n]_{(-1)^{n-1}q},
\end{equation}
where $[k]_q=1+q+\cdots+q^{k-1}=(1-q^k)/(1-q)$.
Adin, Gessel, and Roichman \cite{AGR} extended (\ref{eqn:Gessel-Simion}) to the signed permutation group $B_n$, using the \emph{flag major index} ($\fmaj$) defined by Adin and Roichman \cite{AR},
\begin{equation} \label{eqn:Adin-Gessel-Roichman}
\sum_{\beta\in B_n} \sign(\beta)q^{\fmaj(\beta)}=[2]_{-q}[4]_{q}\cdots [2n]_{(-1)^{n}q}.
\end{equation}

Biagioli and Caselli studied the generating polynomial of the flag major index with each one-dimensional character $\chi_{\epsilon,h}$ over the wreath product $G_{r,n}:=\ZZ_r\wr S_n$ of a cyclic group with the symmetric group, and derived the following formula in the context of  projective reflection group  (cf. \cite[Theorem 4.1]{BC_12}), 
\begin{equation} \label{eqn:Biagioli-Caselli}
\sum_{\pi\in G_{r,n}} \chi_{\epsilon,h}(\pi)q^{\fmaj(\pi)} =[r]_{\zeta^h q}[2r]_{\epsilon\zeta^h q}\cdots [nr]_{\epsilon^{n-1}\zeta^h q},
\end{equation}
where $\epsilon\in\{1,-1\}$, $0\le h\le r-1$ and $\zeta$ is a fixed primitive $r$th root of unity.

For $k<n$, Caselli derived the signed Mahonian over the parabolic quotient of symmetric group $S_n$ by the subgroup $S_k$ \cite[Corollary 3.4]{Caselli_12}, which includes (\ref{eqn:Gessel-Simion}) as a special case,
\begin{equation} \label{eqn:parabolic-Sn}
\sum_{\sigma\in S_n(n-k+1:n)} (-1)^{\inv(\sigma)}q^{\maj(\sigma)} = [k+1]_{(-1)^{nk+n+k}q}[k+2]_{(-1)^{k+1}q}[k+3]_{(-1)^{k+2}q}\cdots [n]_{(-1)^{n-1}q},
\end{equation}
where $S_n(n-k+1:n):=\{\sigma\in S_n:\sigma^{-1}(n-k+1)<\sigma^{-1}(n-k+2)<\cdots<\sigma^{-1}(n)\}$ is a system of representatives of the cosets of $S_k$ in $S_n$. Namely, $S_n(n-k+1:n)\subset S_n$ is the subset of permutations containing the word $(n-k+1,n-k+2,\dots, n)$ as a subsequence. Eu \emph{et al.} derived broader results on the permutations with varied subsequence restrictions \cite[Theorem~1.2]{EFHLS_20}, and suggested extending the results further to the signed permutation group $B_n$, in the spirit of \cite[Problem~5.9]{Caselli_12}.

\smallskip
In the realm of parabolic quotients of $G_{r,n}$, Caselli obtained the distribution
of the flag major index on the quotient of $G_{r,n}$ by the subgroup $G_{r,k}$ \cite[Corollary 5.6]{Caselli_12}, 
\begin{equation} \label{eqn:fmaj-wp}
\sum_{\pi\in C_k} q^{\fmaj(\pi^{-1})}=[(k+1)r]_q[(k+2)r]_q\cdots [nr]_q,
\end{equation}
where $C_k$ is a system of representatives of the cosets of $G_{r,k}$ in $G_{r,n}$. In order to unify (\ref{eqn:parabolic-Sn}) and (\ref{eqn:fmaj-wp}),
Caselli suggested studying further a signed analogue of (\ref{eqn:fmaj-wp}) \cite[Problem 5.9]{Caselli_12},
\begin{equation} \label{eqn:Caselli-problem-5.9}
\sum_{\pi\in C_k} \chi_{\epsilon,h}(\pi^{-1})q^{\fmaj(\pi^{-1})},
\end{equation}
which includes (\ref{eqn:Biagioli-Caselli}) as a special case.  The purpose of this paper is to study the signed Mahonian on the parabolic quotients of $G_{r,n}$ by the subgroups $G_{r,k}$ and $S_k$, respectively, for a variety of systems of coset representatives in terms of subsequence restrictions. See Theorems \ref{thm:Main-result-coset},  \ref{thm:U+V} and \ref{thm:Main-II} for the main results.

\section{Preliminaries and Main Results}
A permutation $\sigma\in S_n$ will be denoted by $\sigma=(\sigma_1,\dots,\sigma_n)$, where $\sigma_i=\sigma(i)$ for all $i\in [n]$. The \emph{inversion number} of $\sigma$ is defined by $\inv(\sigma):=\#\{(i,j):\sigma_i>\sigma_j\mbox{ and } 1\le i<j\le n\}$. 
The set of \emph{descents} of $\sigma$ is $\Des(\sigma):=\{i:\sigma_i>\sigma_{i+1}\mbox{ and } 1\le i\le n-1\}$. 
The \emph{descent number} ($\des$) and \emph{major index} ($\maj$) of $\sigma$ are defined by $\des(\sigma)=|\Des(\sigma)|$ and $\maj(\sigma) = \sum_{i\in\Des(\sigma)} i$, respectively.
We will write $[m,n]:=\{m,m+1,\dots,n\}$ for integers $m< n$ and set $[n]:=[1,n]$.

\subsection{The colored permutation groups.}
The group $G_{r,n}$ of wreath product $\ZZ_r\wr S_n$ consists of all permutations $\pi$ of $[n]\times [0,r-1]$ such that $\pi(a,0)=(b,j)\Rightarrow \pi(a,i)=(b,i+j)$, where $i+j$ is computed modulo $r$, with composition of permutations as the group operation of $\ZZ_r\wr S_n$. The group $G_{r,n}$ reduces to the symmetric group $S_n$ when $r=1$ and the signed permutation group $B_n$ when $r=2$.
Members of $G_{r,n}$ are represented as \emph{colored permutations} on $[n]$ by couples $(\sigma,\z)$, where $\sigma=(\sigma_1,\dots,\sigma_n)\in S_n$, $\z=(z_1,\dots,z_n)\in [0,r-1]^n$ and $z_k$ is the color assigned to $\sigma_k$. We also write $\pi=(\sigma,\z)$ in \emph{window notation} with entries in \emph{colored integers},
\[
\pi=(\pi_1,\pi_2,\dots,\pi_n)=(\sigma_1^{z_1},\sigma_2^{z_2},\cdots,\sigma_n^{z_n}),
\]
where $\pi_i=\pi(i,0)$ for all $i\in [n]$. If $z_i=0$, it is usually omitted. Sometimes $\pi=(\pi_1,\dots,\pi_n)$ is called a \emph{(colored) word}.
Let $|\pi|:=(|\pi_1|,\dots,|\pi_n|)$, where $|\pi_i|=\sigma_i$ is the \emph{absolute value} of $\pi_i$. The \emph{color weight} ($\col$) of $\pi$ is defined by 
\begin{equation} \label{eqn:color-sum}
\col(\pi):=z_1+z_2+\cdots+z_n. 
\end{equation}

% There is a characterization of the length $\ell(\pi)$ of $\pi$ given by Bagno \cite{Bagno}
% \begin{equation} \label{eqn:Bagno}
% \ell(\pi):=\inv(\pi)+\sum_{z_i>0} (\sigma_i+z_i-1),
% \end{equation}
% where $\inv(\pi)$ is computed by usual the linear order
% \begin{equation} \label{eqn:N-order}
% n^{r-1}<\cdots n^1<\cdots<1^{r-1}<\cdots<1^1<0<1<\cdots<n.
% \end{equation}

There are $2r$ one-dimensional characters of the group $G_{r,n}$, which are of the form
\begin{equation} \label{eqn:one-dimension}
\chi_{\epsilon,h}(\pi) := \epsilon^{\inv(|\pi|)}\zeta^{h\cdot \col(\pi)},
\end{equation}
where $\epsilon\in\{1,-1\}$, $h\in [0,r-1]$, and $\zeta$ is a fixed primitive $r$th root of unity. 

% Notice that when $r=2$, we have $G_{2,n}= B_n$,  $\zeta=-1$, and  the four one-dimensional characters $\chi_{1,0}(\pi)=1$, $\chi_{1,1}(\pi)=(-1)^\nega(\pi)$, $\chi_{-1,0}(\pi)=(-1)^{\inv(|\pi|)}$, and $\chi_{-1,1}(\pi)=(-1)^{\ell_B(\pi)}$.
 
\begin{defi} {\rm We use the following order for the flag major index
\begin{equation} \label{eqn:F-linear-order}
F: 1^{r-1}<\cdots<n^{r-1}<\cdots<1^{1}<\cdots<n^{1}<1<\cdots<n.
\end{equation}
}
\end{defi}

The major index ($\maj_F$) and flag major index ($\fmaj$) of $\pi$ are defined by 
\begin{equation} \label{eqn:fmaj-colored}
\maj_F(\pi):=\sum_{\pi_i>\pi_{i+1}} i \qquad\mbox{ and }\qquad \fmaj(\pi) :=r\cdot\maj_F(\pi)+\col(\pi).
\end{equation}
For example, if $r=3$ and $\pi=(4^{2},2^{1},5^{1},1,3^{1})\in G_{3,5}$ then $\col(\pi)=5$, $\maj_F(\pi)=4$, and $\fmaj(\pi)=17$.

\subsection{Main results.}
For two integers $a, b\in [n]$, $a<b$, let $G_{r,n}(a:b)\subset G_{r,n}$
denote the subset of colored words containing $(a,a+1,\dots, b)$ as a subsequence.  We observe that the subgroup of $G_{r,n}$ given by 
\[
\{\pi\in G_{r,n}: \pi=(1^0,\dots,(n-k)^0,\pi_{n-k+1},\dots,\pi_n)\}
\]
is isomorphic to $G_{r,k}$ for all $k<n$, and that 
\[
C_k:=\{(\pi_1,\dots,\pi_{n-k},\sigma_{n-k+1}^0,\dots,\sigma_n^0)\in G_{r,n}:\sigma_{n-k+1}<\cdots<\sigma_n\}
\]
is a system of coset representatives of the subgroup in $G_{r,n}$. Moreover, $G_{r,n}(n-k+1:n)=\{\pi^{-1}:\pi\in C_k\}$.  We obtain the following unified result. 

\begin{thm} \label{thm:Main-result-coset} Let $\chi_{\epsilon,h}$ be a one-dimensional character of $G_{r,n}$. For $1\le k\le n-1$, $k\le b\le n$ and $b\equiv n\pmod 2$, the following results hold.
\begin{enumerate}
\item If $n-k\equiv 0\pmod 2$ then
% \begin{align*}
\[
\sum_{\pi\in G_{r,n}(b-k+1\,:\,b)}\chi_{\epsilon,h}(\pi)q^{\fmaj(\pi)} = [(k+1)r]_{\epsilon^{k}\zeta^h q}\cdots [nr]_{\epsilon^{n-1}\zeta^h q}.
\]
% \end{align*}
\item If $n-k\equiv 1\pmod 2$ then
\[
\sum_{\pi\in G_{r,n}(b-k+1\,:\,b)}\chi_{\epsilon,h}(\pi)q^{\fmaj(\pi)} =\Big([k+1]_{\epsilon q^r}[r]_{\zeta^h q}\Big)[(k+2)r]_{\epsilon^{k+1}\zeta^h q}\cdots [nr]_{\epsilon^{n-1}\zeta^h q}.
\]
\end{enumerate}
\end{thm}

Moreover, we also obtain the signed Mahonian on the remaining cases $G_{r,n}(b-k:b-1)$ with the one-dimensional characters $\chi_{-1,h}$. Neat formulae are unlikely available; however, we provide a formula for computing the signed Mahonian of $G_{r,n}(b-k:b-1)$ from that of $G_{r,n}(b-k+1:b)$.

\begin{thm} \label{thm:U+V} For $1\le k\le n-1$, $k+1\le b\le n$ and $b\equiv n\pmod 2$, let $U$ and $V$ be the words given by
\[
U=(b-k,b-k+1,\dots,b-1) \quad\mbox{ and}\quad V=(b-k+1,b-k+2,\dots,b).
\]
The following results hold.
\begin{enumerate}
\item If $k$ is odd then
\begin{align*}
&\sum_{\pi\in G_{r,n}(U)}\chi_{-1,h}(\pi)q^{\fmaj(\pi)} 
+\sum_{\pi\in G_{r,n}(V)}\chi_{-1,h}(\pi)q^{\fmaj(\pi)} \\
&\qquad=\left\{ \begin{array}{ll}
2[k+1]_{-q^r}[(k+2)r]_{\zeta^h q}[(k+3)r]_{(-1)^{k+2}\zeta^h q}\cdots [nr]_{(-1)^{n-1}\zeta^h q}
&\mbox{for $n$ even;} \\
& \\
2\Big([k+1]_{-q^r}[k+2]_{q^r}[r]_{\zeta^h q} \Big)[(k+3)r]_{(-1)^{k+2}\zeta^h q}\cdots [nr]_{(-1)^{n-1}\zeta^h q}
&\mbox{for $n$ odd.} 
\end{array}
\right.
\end{align*}
\item If $k$ is even then
\begin{align*}
&\sum_{\pi\in G_{r,n}(U)\setminus G_{r,n}(V)} \chi_{-1,h}(\pi)q^{\fmaj(\pi)}+\sum_{\pi\in G_{r,n}(V)\setminus G_{r,n}(U)} \chi_{-1,h}(\pi)q^{\fmaj(\pi)} \\
&\qquad=\left\{ \begin{array}{ll}
2[k+1]_{-q^r}\Big([r]_{\zeta^h q}-1\Big)[(k+2)r]_{(-1)^{k+1}\zeta^h q}\cdots [nr]_{(-1)^{n-1}\zeta^h q}
&\mbox{for $n$ odd;} \\
& \\
2 f(r;q)\cdot [k+1]_{q^r}[k+2]_{-q^r} [(k+3)r]_{(-1)^{k+2}\zeta^h q}\cdots [nr]_{(-1)^{n-1}\zeta^h q}
&\mbox{for $n$ even,} 
\end{array}
\right.
\end{align*}
where 
$f(r;q)=
-\big(\zeta^h q\big)[r-1]_{(-1)^r \zeta^h q}[r]_{(-1)^{r-1}\zeta^h q}$.
\end{enumerate}
\end{thm}

\medskip
\begin{defi} \label{def:|w|} {\rm
For any $w\in S_n$, let $G_{r,n}(w):=\{\pi\in G_{r,n} : |\pi|=w  \}$. For $a,b\in [n]$, $a<b$, we define
\[
H_{r,n}(a:b) := \bigcup_{w\in S_n(a:b)} G_{r,n}(w),
\]  
namely, $H_{r,n}(a:b)$ consists of colored permutations $\pi$ such that $|\pi|$ contains $(a,a+1,\dots,b)$ as a subsequence.
}
\end{defi}

\smallskip
Regarding the signed Mahonian on the parabolic quotient of $G_{r,n}$ by the subgroup $S_k$, we obtain the following result.

\begin{thm}  \label{thm:Main-II} For $1\le k\le n-1$ and $k\le b\le n$,  the following results hold.
\begin{enumerate}
\item If $k$ is odd then 
\[
\sum_{\pi\in H_{r,n}(b-k+1\,:\,b)}\chi_{-1,h}(\pi)q^{\fmaj(\pi)} =\frac{[r]_{\zeta^h q}[2r]_{-\zeta^h q}\cdots[nr]_{(-1)^{n-1}\zeta^h q}}{[1]_{\zeta^h q}[2]_{-\zeta^h q}\cdots[k]_{(-1)^{k-1}\zeta^h q}}.
\]
\item If $k$ is even and $n-b\equiv 0\pmod 2$ then 
% \begin{align*}
\[
\sum_{\pi\in H_{r,n}(b-k+1\,:\,b)}\chi_{-1,h}(\pi)q^{\fmaj(\pi)} 
=\frac{[r]_{\zeta^h q}[2r]_{-\zeta^h q}\cdots [nr]_{(-1)^{n-1}\zeta^h q}}{[1]_{\zeta^h q}[2]_{-\zeta^h q}\cdots[k+1]_{(-1)^{k}\zeta^h q}}\cdot [k+1]_{(-1)^{n}\zeta^h q}.
\]
% \end{align*}
\item If $k$ is even and $n-b\equiv 1\pmod 2$  then 
% \begin{align*}
\[
\sum_{\pi\in H_{r,n}(b-k+1\,:\,b)}\chi_{-1,h}(\pi)q^{\fmaj(\pi)} \\
=\frac{[r]_{\zeta^h q}[2r]_{-\zeta^h q}\cdots [nr]_{(-1)^{n-1}\zeta^h q}}{[1]_{\zeta^h q}[2]_{-\zeta^h q}\cdots[k+1]_{(-1)^{k}\zeta^h q}}\, \Big(2-[k+1]_{(-1)^{n}\zeta^h q}\Big).
\]
% \end{align*}
\end{enumerate}
\end{thm}

The rest of the paper is organized as follows. Section 3 and Section 4 are devoted to the proofs of Theorems \ref{thm:Main-result-coset} and \ref{thm:U+V}, respectively. In Section 5 we derive a factorization formula for the signed Mahonian on the set $G_{r,n}(w)$ (Theorem \ref{thm:Main-I}), which allows us to prove Theorem \ref{thm:Main-II}. 
In Section 6 we obtain a product formula that generalizes Biagioli's result about the signed Mahonian on the even signed permutation groups (Theorem \ref{thm:Main-I-Dmaj}).
Finally, we close with two remarks.

\section{Proof of Theorem \ref{thm:Main-result-coset}}
In this section we shall study the signed Mahonian over the quotient of the parabolic subgroup $G_{r,k}$ in $G_{r,n}$, and prove Theorem \ref{thm:Main-result-coset}. 

\smallskip
For any finite set $A=\{a_1,\dots,a_n\}$ of positive integers, let $S_A$ denote the set of permutations of $\{a_1,\dots,a_n\}$ and let $G_{r,A}:=\ZZ_r\wr S_A$, the group of colored permutations of $A$. For any word $W\in G_{r,A}$,  the inversion number of $|W|$ is calculated using the natural order $a_1<\cdots<a_n$ of integers, and the flag major index of $W$ is computed with respect to the following order of colored integers
\[
a_1^{r-1}<\cdots<a_n^{r-1}<\cdots<a_1^{1}<\cdots<a_n^{1}<a_1<\cdots<a_n.
\]

\smallskip
\subsection{Insertion Lemma.} 
The insertion lemma of Han \cite{Han-thesis} and
Haglund--Loehr--Remmel \cite{HLR} describes the increment of major index resulting from the insertion of an additional element into a given permutation, which works just as well for colored integers in the order (\ref{eqn:F-linear-order}). 

\begin{defi} {\rm Given a set $A\subset [n]$ and a colored word $W\in G_{r,A}$, let $G_{r,n}(W)$ denote the subset of $G_{r,n}$ consisting of the members containing the word $W$ as a subsequence. By abuse of notation, we will use $\chi_{\epsilon,h}$ to denote the following statistic of $W$
\[
\chi_{\epsilon,h}(W):=\epsilon^{\inv(|W|)}\zeta^{h\cdot \col(W)}.
\]
Moreover, for any $m\in [n]\setminus A$, let $T(W; m^t)$ denote the set of words obtained from $W$ by inserting $m^t$ for some $t\in [0,r-1]$.
}
\end{defi}

\smallskip
\noindent
{\bf Remarks:} To distinguish $G_{r,n}(W)$ from the notation $G_{r,n}(w)$ in Definition \ref{def:|w|}, we shall use an upper case letter, say $W$, for a colored word restriction.

\medskip
For any $m\in [n]$, let $W$ be a word of $G_{r,A}$, where $A=\{1,\dots,m-1,m+1,\dots,n\}$. Every member of $G_{r,n}(W)$ can be obtained from $W$ by inserting the colored integer $m^t$ for some $t\in [0,r-1]$.
Haglund, Loehr and Remmel \cite[Corollary 4.2]{HLR} proved that no matter what is the relative order of $m^t$ with respect to the entries of $W$
\begin{equation} \label{eqn:HLR-result}
\sum_{\pi\in T(W; m^t)} q^{\maj_F(\pi)}=q^{\maj_F(W)}[n]_q.
\end{equation}
This leads to the following results for $\epsilon=1$.

\begin{thm} \label{thm:epsilon=1} The following results hold.
\begin{enumerate}
\item For any $m\in [n]$ and any word $W\in G_{r,\{1,\dots,m-1,m+1,\dots,n\}}$, we have
\[ 
\sum_{\pi\in G_{r,n}(W)} \chi_{1,h}(\pi)q^{\fmaj(\pi)}=\chi_{1,h}(W)q^{\fmaj(W)}[nr]_{\zeta^h q}.
\]
\item For any $k\in [n-1]$ and any word $W\in G_{r,\{n-k+1,\dots,n\}}$, we have
\[
\sum_{\pi\in G_{r,n}(W)} \chi_{1,h}(\pi)q^{\fmaj(\pi)}=\chi_{1,h}(W)q^{\fmaj(W)}[(k+1)r]_{\zeta^h q}\cdots [nr]_{\zeta^h q}.
\]
\end{enumerate}
\end{thm}

\begin{proof} (i) Note that $\chi_{1,h}(\pi)=\zeta^{h\cdot\col(\pi)}$ and $\col(\pi)=\col(W)+t$.
By (\ref{eqn:fmaj-colored}) and (\ref{eqn:HLR-result}), we have
\begin{align*}
\sum_{\pi\in T(W; m^t)} \chi_{1,h}(\pi)q^{\fmaj(\pi)}
&=\sum_{\pi\in T(W; m^t)}\zeta^{h\cdot\col(\pi)}q^{r\cdot\maj_F(\pi)+\col(\pi)}  \\
&=\sum_{\pi\in T(W; m^t)} q^{r\cdot\maj_F(\pi)}\big(\zeta^h q\big)^{\col(\pi)}  \\
&= \big(\zeta^h q\big)^{\col(W)+t}\cdot \big(q^r\big)^{\maj_F(W)} [n]_{q^r}  \\
&= \big(\zeta^h q\big)^{t}\cdot \chi_{1,h}(W)q^{\fmaj(W)} [n]_{q^r}. 
\end{align*}
Hence
\begin{align*}
\sum_{\pi\in G_{r,n}(W)} \chi_{1,h}(\pi)q^{\fmaj(\pi)} &=\sum_{t=0}^{r-1} \left(\sum_{\pi\in T(W;m^t)} \chi_{1,h}(\pi)q^{\fmaj(\pi)} \right) \\
&=\chi_{1,h}(W)q^{\fmaj(W)} [n]_{q^r}\Big(1+\zeta^h q+\cdots+\big(\zeta^h q\big)^{r-1}\Big)\\
&=\chi_{1,h}(W)q^{\fmaj(W)} [nr]_{\zeta^h q},
\end{align*}
where the last identity is due to the fact that $\zeta^r=1$.

(ii) We shall prove the assertion by reverse induction on $k$. The initial case $k=n-1$ is proved in (i) when $m=1$. Suppose the assertion holds for $k\ge j$. Given $W\in G_{r,\{n-j+2,\dots,n\}}$, we partition $G_{r,n}(W)$ into the subsets $G_{r,n}(W')$ for every word $W'$ of the set
\[
T(W):=\bigcup_{t=0}^{r-1} T(W; (n-j+1)^t).
\]
By (i), (\ref{eqn:HLR-result}) and induction hypothesis,
\begin{align*}
\sum_{\pi\in G_{r,n}(W)} \chi_{1,h}(\pi)q^{\fmaj(\pi)} &=\sum_{W'\in T(W)}\left( \sum_{\pi\in G_{r,n}(W')} \chi_{1,h}(\pi)q^{\fmaj(\pi)} \right) \\
&=\sum_{W'\in T(W)}\Big( \chi_{1,h}(W')q^{\fmaj(W')} [(j+1)r]_{\zeta^h q}\cdots[nr]_{\zeta^h q}\Big) \\
&=\left(\sum_{W'\in T(W)} \chi_{1,h}(W')q^{\fmaj(W')}\right)[(j+1)r]_{\zeta^h q}\cdots[nr]_{\zeta^h q} \\
&=\chi_{1,h}(W)q^{\fmaj(W)}[jr]_{\zeta^h q}[(j+1)r]_{\zeta^h q}\cdots[nr]_{\zeta^h q}.
\end{align*}
The results follow.
\end{proof}

\smallskip
\subsection{Extended Insertion Lemma.}  Eu \emph{et al.} \cite[Theorem 1.3]{EFHLS_20} derived the generating polynomial for the increment of major index with sign resulting from inserting an additional pair of consecutive elements in any place of a given permutation.

\begin{defi} {\rm Given $A\subset [n]$ with $m,m+1\notin A$ and a word $W\in G_{r,A}$, let $T(W; m^s,(m+1)^t)$ denote the set of words obtained from $W$ by inserting $m^s$ and $(m+1)^t$ adjacently for some $s,t\in [0,r-1]$.
}
\end{defi}

\smallskip
For any $m\in [n-1]$, let $W=(W_1, \dots,W_{n-2})\in G_{r,A}$, where $A=\{1,\dots,m-1,m+2,\dots,n\}$. Every member of $G_{r,n}(W)$ can be obtained from $W$ by inserting $m^s$ and $(m+1)^t$ for some $s,t\in [0,r-1]$.
Let $F_{r,n}(W)\subset G_{r,n}(W)$ denote the subset of members such that the elements $m,m+1$ are adjacent and have the same color, i.e.,
\[
F_{r,n}(W):=\bigcup_{t=0}^{r-1} T(W; m^t,(m+1)^t).
\]
Given $\pi=(\pi_1,\dots,\pi_n)\in G_{r,n}(W)\setminus F_{r,n}(W)$, let $\pi_i=m^{z_i}$ and $\pi_j=(m+1)^{z_j}$, for some $z_i, z_j\in [0,r-1]$. Notice that $\pi_i,\pi_j$ are either adjacent and $z_i\neq z_j$, or not adjacent. There is an immediate involution $\pi\mapsto \pi'$ on the set $G_{r,n}(W)\setminus F_{r,n}(W)$, with $|\inv(|\pi'|)-\inv(|\pi|)|=1$, $\col(\pi')=\col(\pi)$ and $\maj_F(\pi')=\maj_F(\pi)$, such that $\pi'$ is obtained from $\pi$ by replacing the ordered pair $(\pi_i,\pi_j$) with $((m+1)^{z_i},m^{z_j})$. Since $\chi_{-1,h}(\pi)=(-1)^{\inv(|\pi|)}\zeta^{h\cdot\col(\pi)}$ and $\fmaj(\pi)=r\cdot\maj_F(\pi)+\col(\pi)$, we have 
\begin{equation} \label{eqn:F_r,n}
\sum_{\pi\in G_{r,n}(W)} \chi_{-1,h}(\pi)q^{\fmaj(\pi)}=\sum_{\pi\in F_{r,n}(W)} \chi_{-1,h}(\pi)q^{\fmaj(\pi)}
\end{equation}
The right hand side of (\ref{eqn:F_r,n}) can be derived by using the proof of the result in \cite[Lemma 4.3]{EFHLS_20}. We describe the method below.

Every member of $F_{r,n}(W)$ can be obtained from $W$ by inserting $m^t,(m+1)^t$ adjacently,  for some $t\in [0,r-1]$,  to the left of $W$, between two entries of $W$, or to the right of $W$, i.e., one of the $n-1$ \emph{spaces} of $W$. These spaces are indexed by $0,1,\dots,n-2$ from left to right. Using the order (\ref{eqn:F-linear-order}),  the $j$th space, which is between $W_j$ and $W_{j+1}$ is called a $RL$-\emph{space of $W$ relative to $m^t$} if it satisfies one of the following conditions:
\begin{itemize}
\item $j=n-2$ and $W_{n-2} < m^t$,
\item $j=0$ and $m^t<W_1$
\item $0<j<n$ and $W_j>W_{j+1}>m^t$,
\item $0<j<n$ and $m^t>W_j>W_{j+1}$, or
\item $0<j<n$ and $W_j<m^t<W_{j+1}$.
\end{itemize} 
Any space which is not a $RL$-space is called a $LR$-space (relative to $m^t$). Suppose there are $d$ $RL$-spaces of $W$ relative to $m^t$, we label the $RL$-spaces from right to left with $0,1,\dots,d-1$ and label the $LR$-spaces from left to right with $d, d+1,\dots,n-2$. By the same argument as in the proof of \cite[Lemma 4.3]{EFHLS_20}, we have the following result.

\begin{lem} \label{lem:insertion_colored_pair} If $\pi$ is obtained from $W$ by inserting the pair $(x_1,x_2)$ adjacently at the $j$th space of $W$ $(0\le j\le n-2)$, where $(x_1,x_2)$ is either $(m^t,(m+1)^t)$ or $((m+1)^t,m^t)$ for some $t\in [0,r-1]$, then we have
\begin{equation} \label{eqn:(m,m+1)}
\maj_F(\pi)=\left\{ \begin{array}{ll}
\maj_F(W)+a_j+b_j & \mbox{if $(x_1,x_2)=(m^t,(m+1)^t)$,}\\
\maj_F(W)+a_j+b_j+j+1 & \mbox{if $(x_1,x_2)=((m+1)^t,m^t)$,}
                    \end{array}
            \right.         
\end{equation}
where $a_j$ is the label of the $j$th space and $b_j$ is the number of $RL$-spaces relative to $m^t$ appearing to the right of the $j$th space.
\end{lem}

By Lemma \ref{lem:insertion_colored_pair} and the proof of  \cite[Theorem 1.3]{EFHLS_20}, we have
\begin{equation} \label{eqn:maj_F(W)}
\sum_{\pi\in T(W; m^t,(m+1)^t)} (-1)^{\inv(|\pi|)}q^{\maj_F(\pi)} = (-1)^{\inv(|W|)}q^{\maj_F(W)} [n-1]_{(-1)^{n}q}[n]_{(-1)^{n-1}q}.
\end{equation}
Note that (\ref{eqn:HLR-result}) and (\ref{eqn:maj_F(W)}) hold for each $t\in [0,r-1]$ since the proofs of \cite[Corollary 4.2]{HLR} and \cite[Theorem 1.3]{EFHLS_20} work well for individual relative order of $m^t$ with respect to the entries of $W$. This leads to the following signed analogue of Theorem \ref{thm:epsilon=1} for $\epsilon=-1$.

\begin{thm}  \label{thm:Main-III} The following results hold.
\begin{enumerate}
\item For any $m\in [n-1]$ and any word $W\in G_{r,\{1,\dots,m-1,m+2,\dots,n\}}$, we have
\[
\sum_{\pi\in G_{r,n}(W)}\chi_{-1,h}(\pi)q^{\fmaj(\pi)} =\chi_{-1,h}(W)q^{\fmaj(W)} [(n-1)r]_{(-1)^{n}\zeta^h q}[nr]_{(-1)^{n-1}\zeta^h q}.
\]
\item For $1\le k\le \lfloor\frac{n-1}{2}\rfloor$ and any word $W\in G_{r,\{2k+1,\dots,n\}}$, we have
\[
\sum_{\pi\in G_{r,n}(W)} \chi_{-1,h}(\pi)q^{\fmaj(\pi)}=\chi_{-1,h}(W)q^{\fmaj(W)}[(n-2k+1)r]_{(-1)^{n-2k}\zeta^h q}\cdots [nr]_{(-1)^{n-1}\zeta^h q}.
\]
\end{enumerate}
\end{thm} 

\begin{proof} (i)  Note that $\chi_{-1,h}(\pi)=(-1)^{\inv(|\pi|)}\zeta^{h\cdot\col(\pi)}$ and $\col(\pi)=\col(W)+2t$. By (\ref{eqn:fmaj-colored}) and (\ref{eqn:maj_F(W)}), we have 
\begin{align*}
&\sum_{\pi\in T(W; m^t,(m+1)^t)} \chi_{-1,h}(\pi)q^{\fmaj(\pi)} \\
&\qquad =\sum_{\pi\in T(W; m^t,(m+1)^t)} (-1)^{\inv(|\pi|)}q^{r\cdot\maj_F(\pi)}\big(\zeta^h q\big)^{\col(\pi)}  \\
&\qquad = \big(\zeta^h q\big)^{\col(W)+2t}\cdot (-1)^{\inv(|W|)}\big(q^r\big)^{\maj_F(W)} [n-1]_{(-1)^{n}q^r}[n]_{(-1)^{n-1}q^r} \\
&\qquad = \big(\zeta^h q\big)^{2t}\cdot \chi_{-1,h}(W)q^{\fmaj(W)} [n-1]_{(-1)^{n}q^r}[n]_{(-1)^{n-1}q^r}. 
\end{align*}
Hence
\begin{align*}
&\sum_{\pi\in F_{r,n}(W)} \chi_{-1,h}(\pi)q^{\fmaj(\pi)} \\
&\qquad =\sum_{t=0}^{r-1} \left(\sum_{\pi\in T(W; m^t,(m+1)^t)} \chi_{-1,h}(\pi)q^{\fmaj(\pi)} \right) \\
&\qquad =\chi_{-1,h}(\pi)q^{\fmaj(W)} [n-1]_{(-1)^{n}q^r}[n]_{(-1)^{n-1}q^r} \Big(1+\big(\zeta^h q\big)^2+\cdots+\big(\zeta^h q\big)^{2(r-1)}\Big)\\
&\qquad =\chi_{-1,h}(W)q^{\fmaj(W)} [(n-1)r]_{(-1)^{n}\zeta^h q}[nr]_{(-1)^{n-1}\zeta^h q}.
\end{align*}
Along with (\ref{eqn:F_r,n}), the assertion (i) follows.

(ii) Using a similar argument to that for the proof of Theorem \ref{thm:epsilon=1}(ii), the assertion (ii) can be proved by induction on $k$.
\end{proof}

\subsection{Proof of Theorem \ref{thm:Main-result-coset}.} The following lemma will be used in the proof of Theorem \ref{thm:Main-result-coset}.

\begin{lem} \label{lem:(n-k)-insertion} For $1\le k\le n-1$, let $W$ be the word $(n-k+1,\dots,n)$ in $G_{r,\{n-k+1,\dots,n\}}$. Then
\[
\sum_{t=0}^{r-1}\left(\sum_{W'\in T(W; (n-k)^t)} \chi_{-1,h}(W')q^{\fmaj(W')}\right)= [k+1]_{-q^r}[r]_{\zeta^h q}.
\]
\end{lem}

\begin{proof} For each $t\in [0,r-1]$, notice that $(n-k)^t$ is less than every entry of $W$ and that
the leftmost space of $W$ is the unique $RL$-space of $W$ relative to $(n-k)^t$. Moreover, $\inv(W)=0$, $\col(W)=0$, and $\maj_F(W)=0$. If $W'$ is obtained from $W$ by inserting $(n-k)^t$ at the $j$th space from left to right, $0\le j\le k$, then $\inv(|W'|)=j$, $\col(W')=t$ and $\maj_F(W')=j$. Hence
\begin{align*}
\sum_{W'\in T(W; (n-k)^t)} \chi_{-1,h}(W')q^{\fmaj(W')}
&=\sum_{W'\in T(W; (n-k)^t)}(-1)^{\inv(|W'|)}q^{r\cdot\maj_F(W')}\big(\zeta^h q\big)^{\col(W')} \\
&=\big(\zeta^h q\big)^t \sum_{j=0}^{k} (-1)^j\big(q^r\big)^j \\
&=\big(\zeta^h q\big)^{t} [k+1]_{-q^r}.
\end{align*}
Hence
\[
\sum_{t=0}^{r-1}\left(\sum_{\pi\in T(W; (n-k)^t)} \chi_{-1,h}(\pi)q^{\fmaj(\pi)}\right)=[k+1]_{-q^r}\Big(1+\zeta^h q+\cdots +\big(\zeta^h q\big)^{r-1}\Big).
\]
The result follows.
\end{proof}

For any integer $d$ and a word $W=((\sigma_1,\dots,\sigma_k),\z)\in G_{r,k}$, let $W+d:=((\sigma_1+d,\dots,\sigma_k+d),\z)$, a member of $G_{r,\{d+1,\dots,d+k\}}$. The following result, derived from Theorems \ref{thm:epsilon=1}(i) and \ref{thm:Main-III}(i), will be used in the proof of Theorem \ref{thm:Main-result-coset}.

\begin{cor} \label{cor:W-2} For $1\le k\le n-2$, $k+2\le b\le n$, and any word $W\in G_{r,\{b-k+1,\dots,b\}}$, we have
\[
\sum_{\pi\in G_{r,n}(W-2)}\chi_{
\epsilon,h}(\pi)q^{\fmaj(\pi)}=\sum_{\pi\in G_{r,n}(W)}\chi_{\epsilon,h}(\pi)q^{\fmaj(\pi)}.
\]
\end{cor}

\begin{proof} We first prove the case $\epsilon=-1$. Consider the following two sets $G_{r,\{3,\dots,n\}}(W)$ and $G_{r,\{1,\dots,n-2\}}(W-2)$. The map $U\mapsto V$ from $G_{r,\{3,\dots,n\}}(W)$ onto $G_{r,\{1,\dots,n-2\}}(W-2)$ defined by $V=U-2$ is a bijection 
such that $\inv(|V|)=\inv(|U|)$, $\col(V)=\col(U)$ and $\fmaj(V)=\fmaj(U)$.
We partition $G_{r,n}(W-2)$ into subsets $G_{r,n}(V)$ for all $V\in G_{r,\{1,\dots,n-2\}}(W-2)$. Note that every member of $G_{r,n}(V)$ is obtained from $V$ by inserting $(n-1)^s$ and $n^t$ for some $s,t\in [0,r-1]$. By Theorem \ref{thm:Main-III}(i), we have
\[
\sum_{\pi\in G_{r,n}(V)} \chi_{-1,h}(\pi)q^{\fmaj(\pi)}=\chi_{-1,h}(V)q^{\fmaj(V)}[(n-1)r]_{(-1)^{n}\zeta^h q}[nr]_{(-1)^{n-1}\zeta^h q}.
\]
Moreover, we partition $G_{r,n}(W)$ into subsets $G_{r,n}(U)$ for all $U\in G_{r,\{3,\dots,n\}}(W)$. Note that every member of $G_{r,n}(U)$ is obtained from $U$ by inserting $1^s$ and $2^t$ for some $s,t\in [0,r-1]$. Likewise, we have
\[
\sum_{\pi\in G_{r,n}(U)} \chi_{-1,h}(\pi)q^{\fmaj(\pi)}=\chi_{-1,h}(U)q^{\fmaj(U)}[(n-1)r]_{(-1)^{n}\zeta^h q}[nr]_{(-1)^{n-1}\zeta^h q}.
\]
By the bijection between $G_{r,\{3,\dots,n\}}(W)$ onto $G_{r,\{1,\dots,n-2\}}(W-2)$ mentioned above, we have
\begin{align*}
&\sum_{\pi\in G_{r,n}(W-2)} \chi_{-1,h}(\pi)q^{\fmaj(\pi)} \\
&\qquad=\sum_{V\in G_{r,\{1,\dots,n-2\}}(W-2)}\left(\sum_{\pi\in G_{r,n}(V)} \chi_{-1,h}(\pi)q^{\fmaj(\pi)} \right) \\
&\qquad=\left(\sum_{V\in G_{r,\{1,\dots,n-2\}}(W-2)}\chi_{-1,h}(V)q^{\fmaj(V)} \right)[(n-1)r]_{(-1)^{n}\zeta^h q}[nr]_{(-1)^{n-1}\zeta^h q} \\
&\qquad=\left(\sum_{U\in G_{r,\{3,\dots,n\}}(W)}\chi_{-1,h}(U)q^{\fmaj(U)} \right)[(n-1)r]_{(-1)^{n}\zeta^h q}[nr]_{(-1)^{n-1}\zeta^h q} \\
&\qquad=\sum_{U\in G_{r,\{3,\dots,n\}}(W)}\left(\sum_{\pi\in G_{r,n}(U)} \chi_{-1,h}(\pi)q^{\fmaj(\pi)} \right) \\
&\qquad=\sum_{\pi\in G_{r,n}(W)} \chi_{-1,h}(\pi)q^{\fmaj(\pi)}.
\end{align*}
This proves the case $\epsilon=-1$ of the assertion. Using  Theorem \ref{thm:epsilon=1}(i) and a similar argument, we obtain
\[
\sum_{\pi\in G_{r,n}(W-1)} \chi_{1,h}(\pi)q^{\fmaj(\pi)}=
\sum_{\pi\in G_{r,n}(W)} \chi_{1,h}(\pi)q^{\fmaj(\pi)}.
\]
By iteration, the case $\epsilon=1$ of the assertion is proved. \end{proof}

We now prove Theorem \ref{thm:Main-result-coset}.

\medskip
\noindent
\emph{Proof of Theorem \ref{thm:Main-result-coset}.}   Let $W=(n-k+1,\dots,n)$. Consider the parities of $n$ and $k$.

Case 1: $n-k\equiv 0\pmod 2$.  By Theorems \ref{thm:epsilon=1} and \ref{thm:Main-III}, for $\epsilon\in\{1,-1\}$ we have
\[
\sum_{\pi\in G_{r,n}(W)} \chi_{\epsilon,h}(\pi)q^{\fmaj(\pi)}=[(k+1)r]_{\epsilon^{k}\zeta^h q}\cdots [nr]_{\epsilon^{n-1}\zeta^h q}.
\]

Case 2: $n-k\equiv 1\pmod 2$. For the case $\epsilon=1$, by Theorem \ref{thm:epsilon=1}, we also have
\begin{equation} \label{eqn:case_epsilon=1}
\sum_{\pi\in G_{r,n}(W)} \chi_{1,h}(\pi)q^{\fmaj(\pi)}=[(k+1)r]_{\zeta^h q}[(k+2)r]_{\zeta^h q}\cdots [nr]_{\zeta^h q}.
\end{equation}
For the case $\epsilon=-1$, we partition $G_{r,n}(W)$ into subsets $G_{r,n}(W')$, where $W'$ ranges over all words  obtained from $W$ by inserting $(n-k)^t$ for some $t\in [0,r-1]$. By Theorem \ref{thm:Main-III}(ii), we have
\[
\sum_{\pi\in G_{r,n}(W')} \chi_{-1,h}(\pi)q^{\fmaj(\pi)}=\chi_{-1,h}(W')q^{\fmaj(W')}[(k+2)r]_{(-1)^{k+1}\zeta^h q}\cdots [nr]_{(-1)^{n-1}\zeta^h q}.
\]
Moreover, let $T(W):=\bigcup_{t=0}^{r-1}T(W;(n-k)^t)$. We have
\begin{align*}
&\sum_{\pi\in G_{r,n}(W)} \chi_{-1,h}(\pi)q^{\fmaj(\pi)} \\
&\qquad=\sum_{W'\in T(W)}\left(\sum_{\pi\in G_{r,n}(W')} \chi_{-1,h}(\pi)q^{\fmaj(\pi)} \right) \\
&\qquad=\sum_{W'\in T(W)}\left(\chi_{-1,h}(W')q^{\fmaj(W')}[(k+2)r]_{(-1)^{k+1}\zeta^h q}\cdots [nr]_{(-1)^{n-1}\zeta^h q}\right) \\
&\qquad=\sum_{t=0}^{r-1}\left(\sum_{W'\in T(W; (n-k)^t)} \chi_{-1,h}(W')q^{\fmaj(W')}\right)[(k+2)r]_{(-1)^{k+1}\zeta^h q}\cdots [nr]_{(-1)^{n-1}\zeta^h q}.
\end{align*}
Hence by Lemma \ref{lem:(n-k)-insertion}, we have
\begin{equation} \label{eqn:case_epsilon=-1}
\sum_{\pi\in G_{r,n}(W)} \chi_{-1,h}(\pi)q^{\fmaj(\pi)}=\Big([k+1]_{-q^r}[r]_{\zeta^h q} \Big)[(k+2)r]_{(-1)^{k+1}\zeta^h q}\cdots [nr]_{(-1)^{n-1}\zeta^h q}.
\end{equation}
Combining (\ref{eqn:case_epsilon=1}) and (\ref{eqn:case_epsilon=-1}), for $\epsilon\in\{1,-1\}$ we have
\[
\sum_{\pi\in G_{r,n}(W)} \chi_{\epsilon,h}(\pi)q^{\fmaj(\pi)}=\Big([k+1]_{\epsilon q^r}[r]_{\zeta^h q} \Big)[(k+2)r]_{\epsilon^{k+1}\zeta^h q}\cdots [nr]_{\epsilon^{n-1}\zeta^h q}.
\]
Notice that $[k+1]_{\epsilon q^r}[r]_{\zeta^h q}=[(k+1)r]_{\zeta^h q}$ if $\epsilon =1$.
By Corollary \ref{cor:W-2}, for $1\le d\le\lfloor\frac{k}{2}\rfloor$ we have
\begin{equation} \label{eqn:W-2d}
\sum_{\pi\in G_{r,n}(W-2d)} \chi_{\epsilon,h}(\pi)q^{\fmaj(\pi)}=\sum_{\pi\in G_{r,n}(W-2d+2)} \chi_{\epsilon,h}(\pi)q^{\fmaj(\pi)}.
\end{equation} 
The proof of Theorem \ref{thm:Main-result-coset} is completed.
\qed

\section{Proof of Theorem \ref{thm:U+V}}
\medskip
In this section we study the relation between the signed Mahonian polynomials on the sets $G_{r,n}(b-k+1:b)$ and $G_{r,n}(b-k:b-1)$. In the following we compose permutations right to left.

\begin{pro} \label{pro:W-1} For $1\le k\le n-1$ and $k+1\le b\le n$, let $U$ and $V$ be the words  given by
\[
U=(b-k,b-k+1,\dots,b-1) \quad\mbox{ and}\quad V=(b-k+1,b-k+2,\dots,b).
\]
Then there is a bijection $\pi\mapsto\pi'$ of $G_{r,n}(U)\setminus G_{r,n}(V)$ onto $G_{r,n}(V)\setminus G_{r,n}(U)$ with $\Des(\pi')=\Des(\pi)$, $\col(\pi')=\col(\pi)$ and the following property.
\begin{itemize}
\item If $\pi$ contains the entry $b^t$ for some $t\in [1, r-1]$ then $\pi'$ contains the entry $(b-k)^t$ and $\inv(|\pi'|)-\inv(|\pi|)\equiv k\pmod 2$.
\item Otherwise, $\pi$ ($\pi'$, respectively) contains the entry $b$ ($b-k$, respectively) and $\inv(|\pi'|)-\inv(|\pi|)\equiv k-1\pmod 2$.
\end{itemize}
\end{pro}

\begin{proof} (i) Given a word $\pi\in G_{r,n}(U)\setminus G_{r,n}(V)$ with the entry $b^t$ for some $t\in [1,r-1]$, let
\[
\pi'=\left(\begin{array}{ccccc}
                        b-k & b-k+1  & \cdots & b-1 & b^t \\ 
                        b-k+1 & b-k+2  & \cdots & b & (b-k)^t
                       \end{array}
                 \right)\pi.
\]
Notice that if $b^t$ appears to the right (left, respectively) of $b-1$ ($b-k$, respectively) in $\pi$ then $\inv(|\pi'|)-\inv(|\pi|)=k$ ($-k$, respectively).
Moreover, if $b^t$ appears between $b-j-1$ and $b-j$ for some $j$ ($1\le j\le k-1$) then $\inv(|\pi'|)-\inv(|\pi|)=k-2j$. Hence $\inv(|\pi'|)$ and $\inv(|\pi|)$ have the same (opposite, respectively) parity  if $k$ is even (odd, respectively).

(ii) Note that $G_{r,n}(U)\cap G_{r,n}(V)=G_{r,n}(b-k:b)$. Given a word $\pi\in G_{r,n}(U)\setminus G_{r,n}(V)$ with the entry $b$, notice that $b$ appears to the left of $b-1$. There are two cases.

Case 1. If $b$ appears to the left of $b-k$ in $\pi$, then set
\[
\pi'=\left(\begin{array}{cccccc}
                        b     & b-k  & b-k+1 & b-k+2 &\cdots & b-1 \\ 
                        b-k+1 & b-k  & b-k+2 & b-k+3 &\cdots & b
           \end{array}
     \right)\pi.
\]
Note that $\inv(|\pi'|)-\inv(|\pi|)=1-k$.

Case 2. Otherwise, $b$ appears between $b-j-1$ and $b-j$ for some $j$ ($1\le j\le k-1$). If in particular $j=1$ then set
\[
\pi'=\left(\begin{array}{ccccc}
                        b-k   & \cdots  & b-2  &  b & b-1\\ 
                        b-k+1 & \cdots  & b-1  &  b & b-k
           \end{array}
     \right)\pi,
\]
otherwise $2\le j\le k-1$ and set
\[
\pi' = \left(\begin{array}{cccccccc}
                b-k   & \cdots & b-j-1 &  b   &  b-j & b-j+1  & \cdots &  b-1 \\ 
                b-k+1 & \cdots & b-j  & b-j+1 &  b-k   & 
                 b-j+2  & \cdots & b
                \end{array}
                \right)\pi.
\]
Note that in the former case $\inv(|\pi'|)-\inv(|\pi|)=k-1$, while in the latter case $\inv(|\pi'|)-\inv(|\pi|)=k-2j+1$. Hence $\inv(|\pi'|)$ and $\inv(|\pi|)$ have the same (opposite, respectively) parity if $k$ is odd (even, respectively).

The inverse map $\pi'\mapsto\pi$ can be constructed by composing $\pi'$ with the inverse of the permutation. The assertion follows.
\end{proof}

Now we prove Theorem \ref{thm:U+V}.

\medskip
\noindent
\emph{Proof of Theorem \ref{thm:U+V}.}
By Corollary \ref{cor:W-2},  it suffices to consider $U=(n-k,\dots,n-1)$ and $V=(n-k+1,\dots,n)$.
Notice that $G_{r,n}(U)\cap G_{r,n}(V)= G_{r,n}(n-k:n)$. 

(i) For $k$ odd,  by Proposition \ref{pro:W-1} we have
\begin{align*}
\sum_{W\in T(U;n)}\left(\sum_{\pi\in G_{r,n}(W)} \chi_{-1,h}(\pi)q^{\fmaj(\pi)} \right) &=\sum_{W\in T(V;n-k)}\left(\sum_{\pi\in G_{r,n}(W)} \chi_{-1,h}(\pi)q^{\fmaj(\pi)} \right), \\
\sum_{W\in T(U;n^t)}\left(\sum_{\pi\in G_{r,n}(W)} \chi_{-1,h}(\pi)q^{\fmaj(\pi)} \right) &= -\sum_{W\in T(V;(n-k)^t)}\left(\sum_{\pi\in G_{r,n}(W)} \chi_{-1,h}(\pi)q^{\fmaj(\pi)} \right),
\end{align*}
for each $t\in [1,r-1]$. Hence
\begin{equation} \label{eqn:2V}
\begin{aligned}
\sum_{\pi\in G_{r,n}(U)}\chi_{-1,h}(\pi)q^{\fmaj(\pi)} 
&+\sum_{\pi\in G_{r,n}(V)}\chi_{-1,h}(\pi)q^{\fmaj(\pi)} \\
&\qquad 
=2 \sum_{W\in T(V;n-k)}\left(\sum_{\pi\in G_{r,n}(W)} \chi_{-1,h}(\pi)q^{\fmaj(\pi)} \right).
\end{aligned}
\end{equation}

Case 1. $n$ is even. Then $n-k\equiv 1\pmod 2$. We have
\begin{equation} \label{eqn:4.2}
\begin{aligned}
&\sum_{W\in T(V;n-k)}\left(\sum_{\pi\in G_{r,n}(W)} \chi_{-1,h}(\pi)q^{\fmaj(\pi)} \right) \\
&\qquad =\sum_{W\in T(V;n-k)}\left(\chi_{-1,h}(W)q^{\fmaj(W)}[(k+2)r]_{(-1)^{k+1}\zeta^h q}\cdots[nr]_{(-1)^{n-1}\zeta^h q} \right).
\end{aligned}
\end{equation}
Moreover,
\begin{equation} \label{eqn:4.3}
\sum_{W\in T(V;n-k)} \chi_{-1,h}(W)q^{\fmaj(W)}=1-q^r+\cdots+(-1)^k q^r=[k+1]_{-q^r}.
\end{equation}
By (\ref{eqn:2V}), (\ref{eqn:4.2}) and (\ref{eqn:4.3}), the assertion (i) is proved for $n$ even.

Case 2. $n$ is odd. Then $n-k\equiv 0\pmod 2$. By Theorem \ref{thm:Main-III}, it suffices to consider the members in $G_{r,n}(V)$ containing $(n-k-1)^s$ and $n-k$ adjacently for each $s\in [0,r-1]$. Hence
\begin{equation} \label{eqn:4.4}
\begin{aligned}
&\sum_{W\in T(V;n-k)}\left(\sum_{\pi\in G_{r,n}(W)} \chi_{-1,h}(\pi)q^{\fmaj(\pi)} \right) \\
&\qquad =\sum_{s=0}^{r-1}\left(\sum_{W\in T(V;(n-k-1)^s,n-k)}\left(\sum_{\pi\in G_{r,n}(W)} \chi_{-1,h}(\pi)q^{\fmaj(\pi)} \right)\right) \\
&\qquad =\sum_{s=0}^{r-1}\left(\sum_{W\in T(V;(n-k-1)^s,n-k)}\left(\chi_{-1,h}(W)q^{\fmaj(W)}[(k+3)r]_{(-1)^{k+2}\zeta^h q}\cdots[nr]_{(-1)^{n-1}\zeta^h q}\right)\right). 
\end{aligned}
\end{equation}
Moreover,
\begin{equation} \label{eqn:4.5}
\begin{aligned}
&\sum_{s=0}^{r-1}\left(\sum_{W\in T(V;(n-k-1)^s,n-k)} \chi_{-1,h}(W)q^{\fmaj(W)} \right) \\
&\qquad =\left(\sum_{\tau\in T(V;n-k-1,n-k)} \chi_{-1,h}(\tau)q^{\fmaj(\tau)} \right)\Big(1+\zeta^h q+\cdots+\big(\zeta^h q\big)^{r-1}  \Big) \\
&\qquad =[k+1]_{-q^r}[k+2]_{q^r}[r]_{\zeta^h q}.
\end{aligned}
\end{equation}
By (\ref{eqn:2V}), (\ref{eqn:4.4}) and (\ref{eqn:4.5}), the assertion (i) is proved for $n$ odd.

(ii) For $k$ even,  by Proposition \ref{pro:W-1} we have
\begin{align*}
\sum_{W\in T(U;n)}\left(\sum_{\pi\in G_{r,n}(W)} \chi_{-1,h}(\pi)q^{\fmaj(\pi)} \right) &=-\sum_{W\in T(V;n-k)}\left(\sum_{\pi\in G_{r,n}(W)} \chi_{-1,h}(\pi)q^{\fmaj(\pi)} \right), \\
\sum_{W\in T(U;n^t)}\left(\sum_{\pi\in G_{r,n}(W)} \chi_{-1,h}(\pi)q^{\fmaj(\pi)} \right) &= \sum_{W\in T(V;(n-k)^t)}\left(\sum_{\pi\in G_{r,n}(W)} \chi_{-1,h}(\pi)q^{\fmaj(\pi)} \right),
\end{align*}
for each $t\in [1,r-1]$. Hence
\begin{equation} \label{eqn:2(V-U)}
\begin{aligned}
& \sum_{\pi\in G_{r,n}(U)\setminus G_{r,n}(V)} \chi_{-1,h}(\pi)q^{\fmaj(\pi)} +\sum_{\pi\in G_{r,n}(V)\setminus G_{r,n}(U)} \chi_{-1,h}(\pi)q^{\fmaj(\pi)}\\
&\qquad \qquad
=2 \sum_{t=1}^{r-1}\left(\sum_{W\in T(V;(n-k)^t)}\left(\sum_{\pi\in G_{r,n}(W)} \chi_{-1,h}(\pi)q^{\fmaj(\pi)} \right)\right).
\end{aligned}
\end{equation}

Case 1. $n$ odd. The assertion can be proved by the argument of Case 1 of the proof of (i), regarding the members of $G_{r,n}(V)$ containing the entry $(n-k)^t$ for each $t\in [1, r-1]$.

Case 2. $n$ even. It suffices to consider the members in $G_{r,n}(V)$ containing $(n-k-1)^s$ and $(n-k)^t$ adjacently for each $s\in [0,r-1]$ and $t\in [1,r-1]$. Hence
\begin{equation} \label{eqn:4.7}
\begin{aligned}
&\sum_{W\in T(V;(n-k)^t)}\left(\sum_{\pi\in G_{r,n}(W)} \chi_{-1,h}(\pi)q^{\fmaj(\pi)} \right) \\
&\qquad =\sum_{s=0}^{r-1}\left(\sum_{W\in T(V;(n-k-1)^s,(n-k)^t)}\left(\sum_{\pi\in G_{r,n}(W)} \chi_{-1,h}(\pi)q^{\fmaj(\pi)} \right)\right) \\
&\qquad =\sum_{s=0}^{r-1}\left(\sum_{W\in T(V;(n-k-1)^s,(n-k)^t)}\left(\chi_{-1,h}(W)q^{\fmaj(W)}[(k+3)r]_{(-1)^{k+2}\zeta^h q}\cdots[nr]_{(-1)^{n-1}\zeta^h q}\right)\right). 
\end{aligned}
\end{equation}

For each $s\in [0,r-1]$, there is a bijection $W\mapsto W'$ of $T(V;(n-k-1)^t,(n-k)^t)$ onto $T(V;(n-k-1)^s,(n-k)^t)$ such that
\begin{equation} \label{eqn:s-t}
 \chi_{-1,h}(W')q^{\fmaj(W')} =\left\{ \begin{array} {ll}
\chi_{-1,h}(W)q^{\fmaj(W)} \big(\zeta^h q\big)^{s-t} & \mbox{ if $s\ge t$;} \\
\chi_{-1,h}(W)q^{\fmaj(W)} \Big( -\big(\zeta^h q\big)^{s-t}\Big) & \mbox{ if $s<t$.}
\end{array}
\right.
\end{equation}
The bijection is established as follows.
Let $(n-k)^t$ and $(n-k-1)^t$ appear at the entries $W_j,W_{j+1}$ of $W$ for some $j$. The corresponding word $W'$ is obtained from $W$ according to the following rule. If $(W_j,W_{j+1})$ is an ascent, i.e., $(W_j,W_{j+1})=((n-k-1)^t,(n-k)^t)$, then replace the ordered pair by $((n-k-1)^s,(n-k)^t)$ if $s\ge t$ and by $((n-k)^t,(n-k-1)^s)$ if $s<t$. Moreover, if $(W_j,W_{j+1})$ is descent, i.e., $(W_j,W_{j+1})=((n-k)^t,(n-k-1)^t)$, then replace the ordered pair by $((n-t)^t,(n-k-1)^s)$ if $s\ge t$ and by $((n-k-1)^s,(n-k)^t)$ if $s<t$.

It follows that
\begin{equation} \label{eqn:4.9}
\begin{aligned}
&\sum_{s=0}^{r-1}\left(\sum_{W'\in T(V;(n-k-1)^s,(n-k)^t)} \chi_{-1,h}(W')q^{\fmaj(W')} \right) \\
&\qquad =\left(\sum_{W\in T(V;(n-k-1)^t,(n-k)^t)} \chi_{-1,h}(W)q^{\fmaj(W)} \right) \\
&\qquad\qquad\qquad\times \Big(1+\zeta^h q+\cdots+\big(\zeta^h q\big)^{r-t-1}-\big(\zeta^h q\big)^{-1}-\cdots-\big(\zeta^h q\big)^{-t}  \Big) \\
&\qquad =\left(\sum_{W\in T(V;n-k-1,n-k)} \chi_{-1,h}(W)q^{\fmaj(W)} \right) \big(\zeta^h q\big)^{2t} \\
&\qquad\qquad\qquad\times \Big(1+\zeta^h q+\cdots+\big(\zeta^h q\big)^{r-t-1}-\big(\zeta^h q\big)^{-1}-\cdots-\big(\zeta^h q\big)^{-t}  \Big) \\
&\qquad =[k+1]_{q^r}[k+2]_{-q^r}\big(\zeta^h q\big)^{t}\Big(-1-\zeta^h q-\cdots-\big(\zeta^h q\big)^{t-1}+\big(\zeta^h q\big)^t+\cdots+\big(\zeta^h q\big)^{r-1}  \Big). \\
\end{aligned}
\end{equation}
The assertion follows from (\ref{eqn:2(V-U)}), (\ref{eqn:4.7}) and (\ref{eqn:4.9}), 
where the factor $f(r;q)$ is given by
\[
f(r;q) :=\sum_{t=1}^{r-1} \big(\zeta^h q\big)^{t}\Big(-1-\zeta^h q-\cdots-\big(\zeta^h q\big)^{t-1}+\big(\zeta^h q\big)^t+\cdots+\big(\zeta^h q\big)^{r-1}  \Big).
\]
That $f(r;q)=-\big(\zeta^h q\big)[r-1]_{(-1)^r \zeta^h q}[r]_{(-1)^{r-1}\zeta^ q}$ can be proved by induction on $r$.
\qed

\medskip
Specifically, for $r=1$ and the sign character $\chi_{-1,0}(\sigma)=(-1)^{\inv(\sigma)}$, we derive the following results from Theorems \ref{thm:Main-result-coset} and \ref{thm:U+V}, which have been established in \cite[Theorem~1.2]{EFHLS_20}. We make use of the  notation $[n]_{\pm q}! := [1]_q[2]_{-q}\cdots [n]_{(-1)^{n-1}q}$.

\begin{thm} \label{thm:EFHLS-result} {\rm (Eu \emph{et al.} \cite{EFHLS_20})}  For $1\le k\le n-1$ and $k\le b\le n$, the following results hold.
\begin{enumerate}
\item If $k$ is odd then we have
\[
\sum_{\sigma\in S_n(b-k+1\,:\,b)} (-1)^{\inv(\sigma)}q^{\maj(\sigma)} =\dfrac{[n]_{\pm q}!}{[k]_{\pm q}!}.
\]
\item If $k$ is even and $n-b\equiv 0  \pmod{2}$ then
\[
\sum_{\sigma\in S_n(b-k+1\,:\,b)} (-1)^{\inv(\sigma)}q^{\maj(\sigma)} =\dfrac{[n]_{\pm q}!}{[k+1]_{\pm q}!} \cdot [k+1]_{(-1)^{n}q}.
\]
\item If $k$ is even and $n-b\equiv 1  \pmod{2}$ then
\[
\sum_{\sigma\in S_n(b-k+1\,:\,b)} (-1)^{\inv(\sigma)}q^{\maj(\sigma)} 
=  \dfrac{[n]_{\pm q}!}{[k+1]_{\pm q}!}\, \big(2-[k+1]_{(-1)^n q}\big). 
\]
\end{enumerate}
\end{thm}

\section{Proof of Theorem \ref{thm:Main-II}} In this section we shall study the signed Mahonian on the set $G_{r,n}(w):=\{\pi\in G_{r,n}: |\pi|=w\}$ for every $w\in S_n$, and prove Theorem \ref{thm:Main-II}.

We shall show that each $G_{r,n}(w)$ contains a unique member $\tilde{w}=(\tilde{w}_1,\dots,\tilde{w}_n)$ with the least flag major index, and $\fmaj(\tilde{w})=\maj(w)$. We make use of a hierarchy of $G_{r,n}(w)$, established by a sequence of nested subsets, 
\begin{equation} \label{eqn:hierarchical-structure}
G_{r,n}^{(0)}(w)\subset G_{r,n}^{(1)}(w)\subset \cdots\subset G_{r,n}^{(n)}(w)=G_{r,n}(w),
\end{equation}
where $G_{r,n}^{(0)}(w):=\{\tilde{w}\}$ and for $k=1,2,\dots,n$,
\[
G_{r,n}^{(k)}(w):=\{(x_1,\dots,x_k,\tilde{w}_{k+1},\dots,\tilde{w}_{n}) : x_j\in\{w_j,w_j^1,\dots, w_j^{r-1}\}, j=1,\dots,k\}.
\] 
We observe that the structure (\ref{eqn:hierarchical-structure}) realizes the following factorization.

\begin{thm}  \label{thm:Main-I} For any $w\in S_n$, we have
\[
\sum_{\pi\in G_{r,n}(w)} \chi_{\epsilon,h}(\pi)q^{\fmaj(\pi)} = \epsilon^{\inv(w)}(\zeta^h q)^{\maj(w)}\cdot [r]_{\zeta^h q}[r]_{(\zeta^h q)^2}\cdots [r]_{(\zeta^h q)^n}.
\]
\end{thm}

Given a $w\in S_n$, there is a member $\tilde{w}$ of $G_{r,n}(w)$ with $\fmaj(\tilde{w})=\maj(w)$, which can be determined by the following procedure. That $\tilde{w}$ is the unique member of $G_{r,n}(w)$ with the least flag major index follows from Lemma \ref{lem:ab_fmaj}(i). 

\medskip
\noindent
{\bf Algorithm A.} 
\begin{enumerate}
\item If $w$ contains no descent then $w=(1,2,\dots,n)$ and let $\tilde{w}=(w,(0,\dots,0))$. 
\item Otherwise, with respect to the descents of $w$, we decompose $w$ into $\des(w)+1$ increasing runs, indexed by $0,1,\dots,\des(w)$ from right to left. Then $\tilde{w}$ is obtained from $w$ by assigning a color $t\in [0,r-1]$ to each entry in the run indexed by $j$ if $j\equiv t\pmod r$ for $j=0,1,\dots,\des(w)$.
\end{enumerate}

For example, let $w=(8,9,7,1,6,2,4,3,5)\in S_9$. Note that $w$ has five increasing runs $(8,9)$, $(7)$, $(1,6)$, $(2,4)$ and $(3,5)$, and $\maj(w)=17$. For $r=3$, we have $\tilde{w}=(8^1$, $9^1$, $7$, $1^2$, $6^2$, $2^1$, $4^1$, $3$, $5)$ and $\fmaj(\tilde{w})=17$.  For $r=2$, we have $\tilde{w}=(8$, $9$, $7^{1}$, $1$, $6$, $2^{1}$, $4^{1}$, $3$, $5)$ and $\fmaj(\tilde{w})=17$. We have the following observation.

\begin{lem} \label{lem:tilde_omega_color} For any $w\in S_n$, we have
\begin{enumerate}
\item $\fmaj(\tilde{w})=\maj(w)$;
\item $\chi_{\epsilon,h}(\tilde{w}) = \epsilon^{\inv(w)}\zeta^{h\cdot\maj(w)}$.
\end{enumerate}
\end{lem}

\begin{proof} (i) For $1\le i\le \des(w)+1$, let $a_i$ be the number of entries of the $i$th increasing run of $w$ from left to right, and let $b_i=a_1+\cdots+a_i$. Suppose $\des(w)+1=dr+t$ for some $d$ and $t$, $0\le t\le r-1$. Notice that
\begin{equation} \label{eqn:maj(w)--maj_F(w)}
\begin{aligned} 
\maj(w) &= b_1+b_2+\cdots+b_{\des(w)}\\
\maj_F(\tilde{w}) &= b_t+b_{r+t}+\cdots+b_{(d-1)r+t}.
\end{aligned}
\end{equation}
We observe that the color weight of $\tilde{w}$ is
\begin{equation} \label{eqn:col(w)}
\begin{aligned}
\col(\tilde{w}) &= (t-1)a_1+(t-2)a_2+\cdots+a_{t-1} \\
&\qquad+\sum_{i=1}^d \big( (r-1)a_{(i-1)r+t+1}+(r-2)a_{(i-1)r+t+2}+\dots+a_{(i-1)r+t+r-1}\big) \\
&=b_1+\cdots+b_{t-1}+\sum_{i=1}^d \left( \big(b_{(i-1)r+t+1}-b_{(i-1)r+t}\big)+\cdots+\big(b_{(i-1)r+t+r-1}-b_{(i-1)r+t}\big)\right) \\
&=b_1+\cdots+b_{t-1}+\sum_{i=1}^d \left( \big(b_{(i-1)r+t+1}+\cdots+b_{(i-1)r+t+r-1}\big)-(r-1)b_{(i-1)r+t}\right). \\
\end{aligned}
\end{equation}
By (\ref{eqn:maj(w)--maj_F(w)}) and (\ref{eqn:col(w)}), we have
\[
\fmaj(\tilde{w})=r\cdot\maj_F(\tilde{w})+\col(\tilde{w})=b_1+b_2+\cdots+b_{dr+t-1}=\maj(w).
\]

(ii)  By the assertion (i), we have $\maj(w)=\fmaj(\tilde{w})=r\cdot\maj_F(\tilde{w})+\col(\tilde{w})$, and thus $\maj(w)\equiv\col(\tilde{w})\pmod r$.  Note that $|\tilde{w}|=w$. Hence
\[
\chi_{\epsilon,h}(\tilde{w})=\epsilon^{\inv(|\tilde{w}|)}\zeta^{h\cdot\col(\tilde{w})}=\epsilon^{\inv(w)}\zeta^{h\cdot\maj(w)}.
\]
The results follow.
\end{proof}

Consider the sequence of nested subsets of $G_{r,n}(w)$ in (\ref{eqn:hierarchical-structure}). We define maps $\pi\mapsto\varphi(\pi;k,t)$ from $G_{r,n}^{(k-1)}$ to $G_{r,n}^{(k)}$ by a color increment of $t$ on each entry of the prefix of length $k$ of $\pi$.

\begin{defi} \label{def:varphi} {\rm
For $1\le k\le n$, $t\in [0, r-1]$ and any $\pi=(\sigma,(z_1,\dots,z_n))\in G_{r,n}^{(k-1)}$, let 
\[
\varphi(\pi;k,t):=(\sigma,(z_1+t,\dots,z_k+t,z_{k+1},\dots,z_n)).
\]
}
\end{defi}

For example, if $r=3$ and $\pi=(4^{2},2,5^{1},1,3)$, we have $\varphi(\pi;3,2)=(4^{1},2^{2},5,1,3)$. Note that $G_{r,n}^{(k)}(w)=\{\varphi(\pi;k,t):\pi\in G_{r,n}^{(k-1)}(w)$ and $t=0,1,\dots,r-1\}$. We observe the following connection between $G_{r,n}^{(k-1)}(w)$ and $G_{r,n}^{(k)}(w)$.

\begin{lem}\label{lem:ab_fmaj} For $1\le k\le n$, $t\in [0,r-1]$ and any word $\pi\in G_{r,n}^{(k-1)}(w)$, we have 
\begin{enumerate}
\item $\fmaj(\varphi(\pi;k,t)) = \fmaj(\pi) + k\cdot t$.
\item $\chi_{\epsilon,h}(\varphi(\pi;k,t)) = \chi_{\epsilon,h}(\pi)\zeta^{h\cdot k\cdot t}$.
\end{enumerate}
\end{lem}

\begin{proof} (i) If $t=0$ then $\varphi(\pi;k,t)=\pi$ and the assertion holds. For a fixed $t\ge 1$ and any $\pi=(\pi_1,\dots,\pi_n)=(\sigma, (z_1,\dots,z_n))\in G^{(k-1)}_{r,n}(w)$, note that $\pi_j=\tilde{w}_j$ for $k\le j\le n$. For $1\le i\le k-1$,  we observe that in the order (\ref{eqn:F-linear-order}) if $z_i,z_{i+1}\in [0,r-t-1]$ or $z_i,z_{i+1}\in [r-t,r-1]$ then the position $i$ is a descent of both (neither, respectively) of $\pi$ and $\varphi(\pi;k,t)$; otherwise the position $i$ is a descent of either $\pi$ or $\varphi(\pi;k,t)$. We partition the prefix $\pi_1,\dots,\pi_k$ of $\pi$ into alternating sections of maximal sequences of consecutive entries with a color in $[0,r-t-1]$ and consecutive entries with a color in $[r-t,r-1]$. Let $a_0,a_1,\dots,a_{2d}$ be the lengths of these sections (from left to right), where either $d=0$, or $a_0\ge 0$ and $a_1,\dots,a_{2d}>0$ for some $d\ge 1$. Note that  $a_0+a_1+\cdots+a_{2d}=k$. For $0\le j\le 2d$, let $b_j=a_0+a_1+\cdots+a_j$. Consider the following cases of the prefix $\pi_1,\dots,\pi_k$.

Case 1. $z_k\in [0,r-t-1]$. Then the colors of $\pi_1,\dots,\pi_k$ are of the form
\[ \pi_1,\dots,\pi_k:\,\,
\underbrace{[0,r-t-1]}_{a_0},\underbrace{[r-t,r-1]}_{a_1},\underbrace{[0,r-t-1]}_{a_2},\dots,\underbrace{[r-t,r-1]}_{a_{2d-1}}\underbrace{[0,r-t-1]}_{a_{2d}}.
\]
If the position $k$ is a descent of $\pi$, i.e., $\pi_k=\tilde{w}_k>\tilde{w}_{k+1}=\pi_{k+1}$, then by the construction of $\tilde{w}$ in Algorithm A, we have $z_k=0$ and $z_{k+1}=r-1$. It follows that the position $k$ is also a descent of $\varphi(\pi;k,t)$. Otherwise, the position $k$ is not a descent of $\pi$, i.e., either $k=n$ or $\pi_k<\pi_{k+1}$. In the case $k<n$, we observe that either $\sigma_k<\sigma_{k+1}$ and $z_i=z_{k+1}$, or $\sigma_k>\sigma_{k+1}$ and $z_k=z_{k+1}+1$. Hence the position $k$ is not a descent of $\varphi(\pi;k,t)$ either. We have
\begin{align*}
\fmaj(\varphi(\pi;k,t))-\fmaj(\pi) 
&= \left(r\cdot \sum_{j=1}^d (a_0+a_1+\cdots+a_{2j-1}) +t\cdot\sum_{j=0}^d a_{2j}-(r-t)\cdot\sum_{j=1}^d a_{2j-1} \right)\\
&\qquad -r\cdot\sum_{j=0}^{d-1} (a_0+a_1+\cdots+a_{2j}) \\
&=t(a_0+a_1+\cdots+a_{2d}) \\
&=k\cdot t.
\end{align*}

Case 2.  $z_k\in [r-t,r-1]$. Then the colors of $\pi_1,\dots,\pi_k$ are of the form
\[ \pi_1,\dots,\pi_k:\,\,
\underbrace{[r-t,r-1]}_{a_0},\underbrace{[0,r-t-1]}_{a_1},\underbrace{[r-t,r-1]}_{a_2},\dots,\underbrace{[0,r-t-1]}_{a_{2d-1}}\underbrace{[r-t,r-1]}_{a_{2d}}.
\]
Note that $k<n$ and that either $\sigma_k<\sigma_{k+1}$ and $z_k=z_{k+1}$, or $\sigma_k>\sigma_{k+1}$ and $z_k=z_{k+1}+1$. Hence the position $k$ is not a descent of $\pi$. Since $z_k\in [r-t,r-1]$, $z_k+t\equiv z_k+t-r\pmod r$ and $z_k+t-r<z_k$. Hence the position $k$ is always a descent of $\varphi(\pi;k,t)$. We have
\begin{align*}
\fmaj(\varphi(\pi;k,t))-\fmaj(\pi) 
&= \left(r\cdot \sum_{j=1}^d (a_0+a_1+\cdots+a_{2j}) +t\cdot\sum_{j=1}^d a_{2j-1}-(r-t)\cdot\sum_{j=0}^d a_{2j} \right)\\
&\qquad -r\cdot\sum_{j=1}^{d-1} (a_0+a_1+\cdots+a_{2j-1}) \\
&=t(a_0+a_1+\cdots+a_{2d}) \\
&=k\cdot t.
\end{align*}
The assertion (i) follows.

(ii) Let $\pi'=\varphi(\pi;k,t)$. Note that $|\pi'|=|\pi|$ and $\col(\pi')\equiv\col(\pi)+k\cdot t\pmod r$. Hence
\[
\chi_{\epsilon,h}(\varphi(\pi;k,t)) = \epsilon^{\inv(|\pi|)}\zeta^{h(\col(\pi)+k\cdot t)} = \chi_{\epsilon,h}(\pi)\zeta^{h\cdot k\cdot t}.
\]
The assertion (ii) follows.
\end{proof}

\smallskip
Now, we prove Theorem \ref{thm:Main-I}.

\smallskip
\noindent
\emph{Proof of Theorem \ref{thm:Main-I}.} For $0\le k\le n$, we shall prove 
\begin{equation} \label{eqn:induction_on_k}
\sum_{\pi\in G_{r,n}^{(k)}(w)} \chi_{\epsilon,h}(\pi)q^{\fmaj(\pi)}=\epsilon^{\inv(w)}(\zeta^h  q)^{\maj(w)}\cdot
[r]_{\zeta^h q}[r]_{(\zeta^h q)^2}\cdots [r]_{(\zeta^h q)^k}
\end{equation}
by induction on $k$. By Lemma \ref{lem:tilde_omega_color}, for $k=0$ we have 
\begin{equation} \label{eqn:(zeta^kq)^maj}
\sum_{\pi\in G_{r,n}^{(0)}(w)} \chi_{\epsilon,h}(\pi)q^{\fmaj(\pi)}=\chi_{\epsilon,h}(\tilde{w})q^{\fmaj(\tilde{w})}=\epsilon^{\inv(w)}(\zeta^h  q)^{\maj(w)}.
\end{equation}
Moreover, by Lemma \ref{lem:ab_fmaj}, for $k\ge 1$ we have
\begin{align*}
\sum_{\pi\in G_{r,n}^{(k)}(w)}\chi_{\epsilon,h}(\pi)q^{\fmaj(\pi)} &= \sum_{\pi\in G_{r,n}^{(k-1)}(w)} \left( \sum_{j=0}^{r-1}\chi_{\epsilon,h}(\varphi(\pi;k,j))q^{\fmaj(\varphi(\pi;k,j))} \right) \\
&= \sum_{\pi\in G_{r,n}^{(k-1)}(w)} \left( \sum_{j=0}^{r-1}\chi_{\epsilon,h}(\pi) \zeta^{h\cdot k\cdot j}q^{\fmaj(\pi)+k\cdot j} \right) \\
&= \left(1+ \big(\zeta^h q\big)^{k}+\cdots+ \big(\zeta^h q\big)^{(r-1)k}\right) \sum_{\pi\in G_{r,n}^{(k-1)}(w)} \chi_{\epsilon,h}(\pi) q^{\fmaj(\pi)} \\
&=[r]_{(\zeta^h q)^k} \sum_{\pi\in G_{r,n}^{(k-1)}(w)} \chi_{\epsilon,h}(\pi) q^{\fmaj(\pi)}.
\end{align*}
By induction hypothesis, (\ref{eqn:induction_on_k}) follows. The proof of Theorem \ref{thm:Main-I} is completed.
\qed

\medskip
Using (\ref{eqn:Gessel-Simion}) and Theorems \ref{thm:EFHLS-result} and \ref{thm:Main-I}, we now prove Theorem \ref{thm:Main-II}. 

\medskip
\noindent
\emph{Proof of Theorem \ref{thm:Main-II}.}
By Theorem \ref{thm:Main-I}, we have
\begin{equation} \label{eqn:(q;-q)product}
\begin{aligned}
&\sum_{\pi\in H_{r,n}(b-k+1:b)}  \chi_{-1,h}(\pi)q^{\fmaj(\pi)} \\
&\qquad = \sum_{w\in S_n(b-k+1:b)} \left( \sum_{\pi\in G_{r,n}(w)}  \chi_{-1,k}(\pi)q^{\fmaj(\pi)} \right) \\
&\qquad= \left( \sum_{w\in S_n(b-k+1:b)} (-1)^{\inv(w)}\big(\zeta^h q\big)^{\maj(w)}\right)[r]_{\zeta^h q}[r]_{(\zeta^h q)^2}\cdots [r]_{(\zeta^h q)^n}. 
\end{aligned}
\end{equation}
By (\ref{eqn:(q;-q)product}) and Theorem \ref{thm:EFHLS-result}, we have the following results.

(i) If $k$ is odd then 
\begin{align*}
\sum_{\pi\in H_{r,n}(b-k+1:b)}  \chi_{-1,k}(\pi)q^{\fmaj(\pi)}
&=\dfrac{[n]_{\pm \zeta^h q}!}{[k]_{\pm \zeta^h q}!} \cdot [r]_{\zeta^h q}[r]_{(\zeta^h q)^2}\cdots [r]_{(\zeta^h q)^n} \\
&=\frac{[r]_{\zeta^h q}[2r]_{-\zeta^h q}\cdots [nr]_{(-1)^{n-1}\zeta^h q}}{[1]_{\zeta^h q}[2]_{-\zeta^h q}\cdots [k]_{(-1)^{k-1}\zeta^h q}}. 
\end{align*}

%\smallskip
(ii) If $k$ is even and $n-b\equiv 0  \pmod{2}$ then
\begin{align*}
\sum_{\pi\in H_{r,n}(b-k+1:b)}  \chi_{-1,k}(\pi)q^{\fmaj(\pi)}
&=\dfrac{[n]_{\pm \zeta^h q}!}{[k+1]_{\pm \zeta^h q}!} \cdot [k+1]_{(-1)^{n}\zeta^h} \cdot [r]_{\zeta^h q}[r]_{(\zeta^h q)^2}\cdots [r]_{(\zeta^h q)^n}\\
&=\frac{[r]_{\zeta^h q}[2r]_{-\zeta^h q}\cdots [nr]_{(-1)^{n-1}\zeta^h q}}{[1]_{\zeta^h q}[2]_{-\zeta^h q}\cdots [k+1]_{(-1)^{k}\zeta^h q}}\cdot [k+1]_{(-1)^n\zeta^h q}. 
\end{align*}

%\smallskip
(iii) If $k$ is even and $n-b\equiv 1  \pmod{2}$ then
\begin{align*}
\sum_{\pi\in H_{r,n}(b-k+1:b)}  \chi_{-1,k}(\pi)q^{\fmaj(\pi)}
&=\dfrac{[n]_{\pm \zeta^h q}!}{[k+1]_{\pm \zeta^h q}!}\, \big(2-[k+1]_{(-1)^n q}\big) \cdot [r]_{\zeta^h q}[r]_{(\zeta^h q)^2}\cdots [r]_{(\zeta^h q)^n}\\
&=\frac{[r]_{\zeta^h q}[2r]_{-\zeta^h q}\cdots [nr]_{(-1)^{n-1}\zeta^h q}}{[1]_{\zeta^h q}[2]_{-\zeta^h q}\cdots [k+1]_{(-1)^{k}\zeta^h q}}\big(2- [k+1]_{(-1)^n\zeta^h q}\big). 
\end{align*}
The results are established.
\qed

%\section{Signed Mahonian on a subgroup of $G_{r,n}$ }

\section{A Byproduct}
In this section we study the signed Mahonian
over the subgroup of $G_{r,n}$ given by
\begin{equation} \label{eqn:G*}
G^*_{r,n}:=\{\pi\in G_{r,n}: \col(\pi)\equiv 0\pmod r\},
\end{equation}
which is a special complex reflection group \cite{BC_12}, denoted by $G(r,r,n)$. By (\ref{eqn:one-dimension}) and (\ref{eqn:G*}), it is known that the group $G^*_{r,n}$ has 2 one-dimensional characters, $\epsilon^{\inv(|\pi|)}$ for $\epsilon\in\{1,-1\}$.
For a $\pi=(\pi_1,\dots,\pi_n)\in G^*_{r,n}$, let $\Dmaj(\pi)$ denote the statistic of $\pi$ defined by 
\begin{equation} \label{eqn:Dmaj-G*}
\Dmaj(\pi):=\fmaj((\pi_1,\dots,\pi_{n-1},|\pi_n|)),
\end{equation}
following the notion of $D$-major index for the even signed permutation groups defined by Biagioli and Caselli \cite{BC_04}. For any $w\in S_n$, define
\[
G^*_{r,n}(w):=\{\pi\in G^*_{r,n}: |\pi|=w\}.
\]
We obtain the following results. 

\begin{thm}  \label{thm:Main-I-Dmaj}  For $\epsilon\in\{1,-1\}$, the following results hold.
\begin{enumerate}
\item For any $w\in S_n$, we have
\[
\sum_{\pi\in G^*_{r,n}(w)} \epsilon^{\inv(|\pi|)}q^{\Dmaj(\pi)} = \epsilon^{\inv(w)}q^{\maj(w)}\cdot[r]_{q}[r]_{q^2}\cdots [r]_{q^{n-1}}.
\]
\item We have
\[
\sum_{\pi\in G^*_{r,n}} \epsilon^{\inv(|\pi|)}q^{\Dmaj(\pi)} = [r]_{q}[2r]_{\epsilon q}\cdots [(n-1)r]_{\epsilon^n q}[n]_{\epsilon^{n-1} q}.
\]
\end{enumerate}
\end{thm}

Note that for $r=2$ in Theorem \ref{thm:Main-I-Dmaj}(ii), we obtain \cite[Theorem 4.8]{Biagioli_06}. 
For any word $\pi=(\sigma,(z_1,\dots,z_n))\in G_{r,n}$, let $\pi^*\in G^*_{r,n}$ denote the word given by
\begin{equation} \label{eqn:r-colored-Dmaj}
\pi^*:=(\sigma, (z_1,\dots,z_{n-1},z^*_n)), 
\end{equation}
where $z^*_n\equiv -(z_1+\cdots+z_{n-1})\pmod r$.
For any $w\in S_n$, let $\tilde{w}=(\tilde{w}_1,\tilde{w}_2,\dots,\tilde{w}_n)\in G_{r,n}(w)$ be the word obtained from $w$ by Algorithm A, and let $\tilde{w}^*$ be determined from $\tilde{w}$ by (\ref{eqn:r-colored-Dmaj}). Note that $\tilde{w}^*$ is the unique member of $G^*_{r,n}(w)$ with the minimum $\Dmaj$ by Lemma~\ref{lem:cs_Dmaj}.

\begin{lem} \label{lem:Dmaj-maj} For any $w\in S_n$, we have  $\Dmaj(\tilde{w}^*)=\maj(w)$.
\end{lem}

\begin{proof}  By (\ref{eqn:Dmaj-G*}), (\ref{eqn:r-colored-Dmaj}) and Lemma \ref{lem:tilde_omega_color}, we have $\Dmaj(\tilde{w}^*)=\fmaj(\tilde{w})=\maj(w)$.
\end{proof}

We associate $\tilde{w}^*$ with a sequence of nested subsets of $G^*_{r,n}(w)$,
\[
G^{*(0)}_{r,n}(w)\subset G^{*(1)}_{r,n}(w)\subset\cdots\subset G^{*(n-1)}_{r,n}(w)=G^*_{r,n}(w),
\]
where $G^{*(0)}_{r,n}(w)=\{\tilde{w}^*\}$ and for $k=1,2,\dots,n-1$,
\[
G^{*(k)}_{r,n}(w)=\{(x_1,\dots,x_k,\tilde{w}_{k+1},\dots,\tilde{w}_{n-1}, w_n^{z^*_n}) : x_j\in\{{w}_j,{w}^1_j,\dots, {w}^{r-1}_j\}, j=1,\dots,k\}.
\] 
By the same argument as in the proof of Lemma \ref{lem:ab_fmaj}, we have the following connection between $G^{*(k-1)}_{r,n}(w)$ and $G^{*(k)}_{r,n}(w)$.

\begin{lem}\label{lem:cs_Dmaj}
For $1\le k\le n-1$, $t\in [0,r-1]$ and any word $\pi\in G^{*(k-1)}_{r,n}(w)$,  we have 
\[
\Dmaj(\varphi(\pi;k,t)^*) = \Dmaj(\pi) + k\cdot t.
\]
\end{lem}

\begin{proof} By (\ref{eqn:Dmaj-G*}), (\ref{eqn:r-colored-Dmaj}) and Lemma \ref{lem:ab_fmaj}, we have
\[
\Dmaj(\varphi(\pi;k,t)^*)=\fmaj( \varphi(\pi;k,t) )= \fmaj(\pi)+k\cdot t=\Dmaj(\pi) + k\cdot t.
\]
The result follows.
\end{proof}

\medskip
\noindent
\emph{Proof of Theorem \ref{thm:Main-I-Dmaj}.} (i) By Lemma \ref{lem:Dmaj-maj}, we have
\begin{equation} \label{eqn:signed-D_n^0}
\sum_{\pi\in G^{*(0)}_{r,n}(w)} \epsilon^{\inv(|\pi|)} q^{\Dmaj(\pi)}=\epsilon^{\inv(|\tilde{w}^*|)}q^{\Dmaj(\tilde{w}^*)}=\epsilon^{\inv(w)}q^{\maj(w)}.
\end{equation}
Moreover, by (\ref{eqn:signed-D_n^0}) and Lemma \ref{lem:cs_Dmaj}, we have
\begin{align*}
\sum_{\pi\in G^{*(k)}_{r,n}(w)} \epsilon^{\inv(|\pi|)}q^{\Dmaj(\pi)}
&= \sum_{\pi\in G^{*(k-1)}_{r,n}(w)} \left( \sum_{j=0}^{r-1} \epsilon^{\inv(|\varphi(\pi;k,j)^*|)}q^{\Dmaj(\varphi(\pi;k,j)^*)} \right) \\
&= \big(1+q^k+\cdots+q^{(r-1)k}\big)\sum_{\pi\in G^{*(k-1)}_{r,n}(w)} \epsilon^{\inv(|\pi|)}q^{\Dmaj(\pi)} \\
&=\left(\prod_{i=1}^{k} \big(1+q^i+\cdots+q^{(r-1)i}\big) \right)\sum_{\pi\in G^{*(0)}_{r,n}(w)} \epsilon^{\inv(|\pi|)}q^{\Dmaj(\pi)} \\
&=\epsilon^{\inv(w)}q^{\maj(w)}\cdot\prod_{i=1}^{k} \big(1+q^i+\cdots+q^{(r-1)i}\big).
\end{align*}
By induction, the assertion (i) of Theorem \ref{thm:Main-I-Dmaj} follows.

(ii) By (\ref{eqn:Gessel-Simion}) and the assertion (i), we have
\begin{align*}
\sum_{\pi\in G^*_{r,n}}\epsilon^{\inv(|\pi|)}q^{\Dmaj(\pi)} &= \sum_{w\in S_n} \left(\sum_{\pi\in G^*_{r,n}(w)}\epsilon^{\inv(|\pi|)}q^{\Dmaj(\pi)}\right) \\
&= \left(\prod_{i=1}^{n-1} \big( 1+ q^i+\cdots+ q^{(r-1)i}\big)\right)  \sum_{w\in S_n} \epsilon^{\inv(w)} q^{\maj(w)} \\
&= \left(\prod_{i=1}^{n-1} \big( 1+ q^i+\cdots+ q^{(r-1)i}\big)\right)  [1]_{q} [2]_{\epsilon q} \cdots  [n]_{\epsilon^{n-1} q} \\
&= [r]_{q} [2r]_{\epsilon q}\cdots [(n-1)r]_{\epsilon^{n-2} q}[n]_{\epsilon^{n-1} q}.
\end{align*}
The proof of Theorem \ref{thm:Main-I-Dmaj} is completed.
\qed

\section{Two Remarks}
Biagioli and Caselli derived the result (\ref{eqn:Biagioli-Caselli}) in the context of projective reflection group (cf. \cite[Theorem~4.1]{BC_12}.
One of the key ingredients of Biagioli and Caselli's method is to prove the following formula \cite[Theorem 4.4]{BC_12}
\begin{equation} \label{eqn:ingredients}
\sum_{\pi\in G_{r,n}} (-1)^{\inv(|\pi|)} q^{\fmaj(\pi)}=[r]_q[2r]_{-q}\cdots [nr]_{(-1)^{n-1}q},
\end{equation}
using the decomposition
% \begin{equation}
$G_{r,n}=U_n S_n$,
% \end{equation}
where $U_n:=\{\pi\in G_{r,n}: \pi_1<\cdots <\pi_n\}$, i.e., every $\pi\in G_{r,n}$ has a unique factorization $\pi=\tau\sigma$ with $\tau\in U_n$ and $\sigma\in S_n$. This approach is also used in \cite{ABR,AGR,BZ}. See  \cite[Proposition 4.1]{BZ} for a proof. Using  (\ref{eqn:Gessel-Simion}),  the following intermediate stage is deduced
\begin{equation} \label{eqn:parabolic_decomposition}
\sum_{\pi\in G_{r,n}} (-1)^{\inv(|\pi|)} q^{\fmaj(\pi)}=\sum_{\tau\in U_n} (-1)^{\inv(|\tau|)} q^{\col(\tau)} [1]_{q^r}[2]_{-q^r}\cdots [n]_{(-1)^{n-1}q^r}.
\end{equation}
%However, the computation of $\sum_{\tau\in U_n} \sign(|\tau|) q^{\col(\tau)}$ is quite involved.
We present an alternative proof of (\ref{eqn:Biagioli-Caselli}), using (\ref{eqn:Gessel-Simion}) and Theorem \ref{thm:Main-I}, as follows.

\begin{align*}
\sum_{\pi\in G_{r,n}}\chi_{\epsilon,h}(\pi)q^{\fmaj(\pi)} &= \sum_{w\in S_n} \left(\sum_{\pi\in G_{r,n}(w)}\chi_{\epsilon,h}(\pi)q^{\fmaj(\pi)}\right) \\
&= [r]_{\zeta^h q}[r]_{(\zeta^h q)^2}\cdots [r]_{(\zeta^h q)^n}  \sum_{w\in S_n} \epsilon^{\inv(w)} (\zeta^h q)^{\maj(w)} \\
&= [r]_{\zeta^h q}[r]_{(\zeta^h q)^2}\cdots [r]_{(\zeta^h q)^n}\cdot  [1]_{\zeta^h q} [2]_{\epsilon\zeta^h q} \cdots  [n]_{\epsilon^{n-1}\zeta^h q} \\
&= [r]_{\zeta^h q} [2r]_{\epsilon\zeta^h q}\cdots [nr]_{\epsilon^{n-1}\zeta^h q}.
\end{align*}
The proof of (\ref{eqn:Biagioli-Caselli}) is completed.

In this paper, we study the signed Mahonian on the quotients of the parabolic subgroup $G_{r,k}$ of $G_{r,n}$. We observe that two systems $T$ and $T'$ of coset representatives of $G_{r,k}$ may share the same signed Mahonian polynomials. We write $T\sim T'$ if the following relation between $T$ and $T'$ holds. 
\begin{align*}
\sum_{\pi\in T}\chi_{\epsilon,h}(\pi)q^{\fmaj(\pi)} = \sum_{\pi\in T'}\chi_{\epsilon,h}(\pi)q^{\fmaj(\pi)}.
\end{align*} 
Theorem \ref{thm:Main-result-coset} indicates that $G_{r,n}(b-k-1:b-2)\sim G_{r,n}(b-k+1:b)$ for each one-dimensional character, and Theorem \ref{thm:Main-II} implies that for $\epsilon=-1$, $H_{r,n}(b-k:b-1)\sim H_{r,n}(b-k+1:b)$ if $k$ is odd and $H_{r,n}(b-k-1:b-2)\sim H_{r,n}(b-k+1:b)$ if $k$ is even.

For any word $\alpha=(\alpha_1,\ldots,\alpha_k)$ on the set $[n]$, let $S_n(\alpha)\subset S_n$ be the subset of permutations containing $\alpha$ as a subsequence, and let
\[
H_{r,n}(\alpha):=\bigcup_{w\in S_n(\alpha)} G_{r,n}(w).
\] 

For $\epsilon=1$, using Theorem 5.1, we have
\begin{align*}
\sum_{\pi\in H_{r,n}(\alpha)}\chi_{1,h}(\pi)q^{\fmaj(\pi)} &= \sum_{w\in S_n(\alpha)} \left( \sum_{\pi\in G_{r,n}(w)}\chi_{1,h}(\pi)q^{\fmaj(\pi)} \right) \\
&= [r]_{\zeta^hq}[r]_{(\zeta^hq)^2}\cdots[r]_{(\zeta^hq)^n} \sum_{w\in S_n(\alpha)}\left(\zeta^hq\right)^{\maj(w)} \\
&= [r]_{\zeta^hq}[r]_{(\zeta^hq)^2}\cdots[r]_{(\zeta^hq)^n}\cdot \left(\zeta^hq\right)^{\maj(\alpha)}[k+1]_{\zeta^hq}[k+2]_{\zeta^hq}\cdots[n]_{\zeta^hq} \\
&= \left(\zeta^hq\right)^{\maj(\alpha)} \cdot \frac{[r]_{\zeta^hq}[2r]_{\zeta^hq}\cdots[nr]_{\zeta^hq}}{[1]_{\zeta^hq}[2]_{\zeta^hq}\cdots[k]_{\zeta^hq}}.
\end{align*}
Hence for any word $\alpha=(\alpha_1,\ldots,\alpha_k)$ on the set $[n-1]$, we actually have $H_{r,n}(\alpha)\sim H_{r,n}(\alpha+1)$ for $\epsilon=1$. We are interested in a bijective proof of the above relations.

\section*{Acknowledgements.}
The authors thank the referees for reading the
manuscript carefully and providing helpful suggestions that improve the presentation of the paper. 
The authors were supported in part by
Ministry of Science and Technology (MOST) grant 107-2115-M-003-009-MY3 (S.-P. Eu), 109-2115-M-153-004-MY2 (T.-S. Fu), and 108-2115-M-153-004-MY2 (Y.-H. Lo).

\end{document}